\documentclass[a4paper,10pt]{scrartcl}
\usepackage[fleqn]{amsmath}
\usepackage{amssymb}
\usepackage{amsthm}
\usepackage{stmaryrd}
\usepackage{bbm}
\usepackage{hyperref}
\usepackage{mathtools}
\usepackage[numbers]{natbib}

\mathtoolsset{showonlyrefs}

\newcommand{\bE}{\ensuremath{\mathbb{E}}}

\newcommand{\bP}{\ensuremath{\mathbb{P}}}

\newcommand{\bR}{\ensuremath{\mathbb{R}}}

\newcommand{\bZ}{\ensuremath{\mathbb{Z}}}

\newcommand{\ind}{\ensuremath{\mathbbm{1}}}

\newcommand{\cC}{\ensuremath{\mathcal{C}}}

\newcommand{\norm}[1]{\left\Vert \, #1 \, \right\Vert}

\newcommand{\ddx}[1][1]{\ifnum#1=1 \frac{d}{dx} \else \frac{d^{#1}}{dx^{#1}} \fi}
\newcommand{\ddy}[1][1]{\ifnum#1=1 \frac{d}{dy} \else \frac{d^{#1}}{dy^{#1}} \fi}
\newcommand{\ddt}[1][1]{\ifnum#1=1 \frac{d}{dt} \else \frac{d^{#1}}{dt^{#1}} \fi}

\newcommand{\suml}{\sum\limits}

\usepackage{todonotes} 

\theoremstyle{plain}
\newtheorem{thm}{Theorem}[section]  
\newtheorem{prop}[thm]{Proposition}

\newtheorem{lem}[thm]{Lemma}

\theoremstyle{definition}

\theoremstyle{remark}
\newtheorem{rem}{Remark}[section]

\numberwithin{equation}{section}

\DeclareMathOperator{\integers}{\mathbb{Z}}
\DeclareMathOperator{\reals}{\mathbb{R}}

\DeclareMathOperator{\nat}{\mathbb{N}}

\newcommand{\Latd}{\mathbb{Z}^d}

\newcommand{\hP}{\widehat{\bP}}

\begin{document}

\title{The contact process as seen from a random walk}
\author{Stein Andreas Bethuelsen\footnote{Email: bethuelsensa@math.leidenuniv.nl} \\ Leiden University and CWI Amsterdam }

\maketitle
\begin{abstract}
We consider a random walk on top of the contact process on $\mathbb{Z}^d$ with $d\geq 1$. In particular, we focus on the ``contact process as seen from the random walk''. Under the assumption  that the infection rate of the contact process is large  or the jump rate of the random walk is small, we show that  this process has at most two extremal measures. Moreover, the  convergence to these extremal measures is characterised by  whether the contact process survives or dies out, similar to the complete convergence theorem known for the ordinary contact process. Using this, we furthermore provide a law of large numbers for the random walk which holds under general assumptions on the jump probabilities of the random walk.
\end{abstract}

\bigskip

 \bigskip

\section{Introduction and main results}

\subsection{Motivation, background and outline}

In this paper we study a random walk on top of the contact process on $\Latd$ with $d\geq 1$.  That is, we assume that the transition kernel of the random walk depends on the  contact process in a local neighbourhood around the position of the random walk.
This is an example of a random walk in a dynamic random environment (abbreviated by RWDRE), a class of models that have recently been the subject of intensive studies in the mathematical literature (see e.g.\ \cite{AndresChiariniDeuschelSlowikRCMDRE2016}, \cite{AvenaBlondelFaggionatoRWDRE2016}, \cite{BirknerCernyDepperschmidtRWDRE2015}, \cite{DeuschelGuoRamirezRWBDRE2015}, \cite{HilarioHollanderSidoraviciusSantosTeixeiraRWRW2014}, \cite{HuveneersSimenhausRWSEP2014}, \cite{RedigVolleringRWDRE2013}). 

The contact process is a classical interacting particle system.  
This model was first introduced by \citet{HarrisCP1974} in the $1970$'s as a model for the spread of an infection in a population. In this model, ``infections'' spread to nearest neighbour sites at a constant rate $\lambda$ and a site become ``healthy'' at a constant rate $1$. 

On the one hand, the contact process is a model to which many of the mathematical tools developed for studying disordered systems apply, such as monotonicity, duality, renormalisation and coupling, and by now much is known. For instance, a full understanding of its dependence on the initial state is known by the complete convergence theorem. As a general reference about the contact process, we refer to \citet[Chapter 1]{LiggettSIS1999}.

On the other hand, the contact process is a complicated model. Indeed, since infections spread in space and time, it has  a non-trivial spatial and temporal correlation structure. Moreover, the contact process has a phase transition. For infection rate $\lambda$ sufficiently small, the whole population eventually becomes healthy, irrespectively of the initial configuration. Interestingly, for infection rate above a certain threshold,  infections may spread for all times with positive probability. In particular, in this regime (called the supercritical regime), the evolution of the contact process depends strongly on the initial configuration.

Furthermore, the contact process is an example of a model which, in the supercritical regime, does not fall into the class of well behaved models characterised by the cone-mixing condition. In particular, the general results obtained by  \citet*{AvenaHollanderRedigRWDRELLN2011} and  \citet{RedigVolleringRWDRE2013} do not apply to random walks on the supercritical contact process. Despite much progress in the last years, no general theory has so far been developed for RWDRE models when the dynamic random environment is not cone mixing.

For the above reasons, the study of a random walk on the supercritical contact process in the context considered in this paper was initiated by \citet{HollanderSantosRWCP2013}. They considered a class of nearest neighbour random walks on the one dimensional contact process. Combining monotonicity properties of the contact process and the random walk, they proved a law of large numbers, valid  throughout the supercritical regime, and, assuming large enough infection rate, a central limit theorem. Since then, the model has been studied in several papers. We mention  in particular \citet{BethuelsenHeydenreichRWAD2015}, who proved a law of large numbers for a version of the model on $\Latd$ with $d\geq1$, and  \citet{MountfordVaresDCP2013}, who improved the  central limit theorem of \citet{HollanderSantosRWCP2013} and proved that it holds throughout the supercritical regime.  See also \citet{BethuelsenVolleringRWDRE2016} and \citet*{BirknerCernyDepperschmidtRWDRE2015} for related results.

In contrast to \cite{BethuelsenHeydenreichRWAD2015}, \cite{HollanderSantosRWCP2013} and \cite{MountfordVaresDCP2013}, who studied the evolution of the random walk directly, the focus of this paper is on the ``contact process as seen from the random walk''-process (abbreviated by CPSRW). That is, we study the shift-perturbed version of the ordinary contact process on $\Latd$, $d\geq1$,  such that the random walk always remain at the origin. 
In particular, we are interested in the set of invariant measures for the CPSRW process and its convergence towards the extremal ones. 

Our main contribution is that, when the infection rate of the contact process is large or the jump rate of the random walk is small, then the CPSRW process satisfies a complete convergence theorem similar to what is known for the ordinary contact process. That is, the CPSRW process has (at most) two extremal invariant measures making the process ergodic and it converges towards a mixture of these states (in the Ces\`aro sense) depending only on whether the underlying contact process survives or dies out. For this, we allow for very general transitions kernels of the random walk. 

As a consequence of this result about the CPSRW process, we also derive limiting properties about the random walk itself. In particular, we show that it satisfies a law of large numbers under rather general assumptions on the transition kernel.

\subsubsection*{Outline}
In the next subsection we give a more precise definition of our model and in particular the CPSRW process. Our main results are presented in Subsection \ref{sec main}. As preparations for the proofs, we provide in Section \ref{sec contact} some preliminary results about the contact process and in Section \ref{sec coupling}  we provide a particular coupling construction of our model. Section \ref{sec results Stein} contains the proofs of our main results.

\subsection{The model}
\subsubsection*{The contact process}

Let $\Omega = \{0,1\}^{\Latd}$. 
For $\eta \in \Omega$ and $x \in \Latd$, we denote by $\eta_x$ the configuration which is identical to $\eta$ except at site $x$, where a $1$ is replaced by a $0$ and vice versa. We also denote by $s(\eta,x) := \sum_{y \sim x} \eta(y)$, where $\sum_{y \sim x}$ is the summation over nearest neighbours of $x$.

The contact process $(\eta_t)_{t\geq0}$ on $\Latd$ with ``infection'' rate $\lambda>0$ and ``recovery'' rate $1$ is the Markov process on $\Omega$ with generator $L \colon \mathcal{C}(\Omega; \reals) \mapsto \mathcal{C}(\Omega; \reals)$, where $ \mathcal{C}(\Omega; \reals)$ denotes the space of bounded continuous functions from $\Omega$ to $\reals$, and $L$ is given by
\begin{align}
Lf(\eta) = \sum_{x \in \Latd} \bigg[ \eta(x) \left[f(\eta_x) - f(\eta) \right] + \lambda s(\eta,x)\left(1-\eta(x)\right)\left[ f(\eta_x)-f(\eta)\right] \bigg].
\end{align}
We denote the semi-group generated by $L$  by $(S_t)_{t\geq0}$,  also considered on the space $\mathcal{C}(\Omega,\reals)$. 
 Note that the contact process is translation invariant, that is,
\begin{align}
 \bP_{\eta,\lambda} (\theta_x \eta_t \in \cdot) = \bP_{\theta_x \eta,\lambda} (\eta_t \in \cdot)
\end{align}
with $\theta_x$ denoting the shift operator $\theta_x \eta(y) = \eta(y-x)$ and $\bP_{\eta,\lambda}$ the path-space measure of the contact process on $D_{\Omega}[0,\infty)$, the set of c\`adl\`ag functions on $[0,\infty)$ taking values on $\Omega$, starting from $\eta_0=\eta$. Further, denote by $\mathcal{F}$ the product $\sigma$-algebra corresponding to $\Omega$ and let $\mathcal{M}_1(\Omega)$ denote the set of probability measures on $(\Omega,\mathcal{F})$. By $\delta_{\eta} \in \mathcal{M}_1(\Omega)$ we denote the measures which concentrates on $\eta \in \Omega$. For $\mu \in \mathcal{M}_1(\Omega)$ we denote by $\bP_{\mu,\lambda}$ the path-space measure of $(\eta_t)_{t \geq 0}$ when the contact process is initialised from $\mu$, that is, $\bP_{\mu,\lambda}(\cdot) := \int \bP_{\eta,\lambda}(\cdot) \mu(d\eta)$. 

The empty configuration where all sites equal to $0$, denoted by $\bar{0}$, is an absorbing state for the contact process since $s(\bar{0},x)=0$ for all $x \in \Latd$. On the other hand, when initialised with all sites equal to $1$, the contact process is known to evolve towards an equilibrium measure  called the \emph{upper invariant measure}. We denote this measure by $\bar{\nu}_{\lambda}$.  

As already mentioned, the contact process has a phase transition. That is, there is a critical threshold $\lambda_c \in (0,\infty)$, where $\lambda_c$ depends on the dimension, such that $\nu_{\lambda}=\delta_{\bar{0}}$ for $\lambda \in (0,\lambda_c]$ and, for all $\lambda \in (\lambda_c,\infty)$, it holds that  $\bar{\nu}_{\lambda}(\eta(x) =1)>0$ for any $x \in \Latd$.
Further, the two measures, $\delta_{\bar{0}}$ and $\bar{\nu}_{\lambda}$, are the only extremal measures for the contact process on $\Latd$. 
A complete description of the convergence  towards any mixture of them is known by the complete convergence theorem, which for later reference we state next.  (For a proof we refer to \cite[Theorem 1.2.27]{LiggettSIS1999}).

\begin{thm}[Complete Convergence for $(\eta_t)$]\label{thm complete convergence}
Let $\tau := \inf \{ t \geq 0 \colon \eta_t = 1_{\bar{0}} \}$. Then, for $\lambda>0$ and $\eta \in \Omega$;
\begin{align}
\delta_{\eta} S_t  \implies \bP_{\eta,\lambda}( \tau <\infty) \delta_{\bar{0}} + \bP_{\eta,\lambda}(\tau = \infty) \bar{\nu}_{\lambda} \quad \text{ as } t \rightarrow \infty,
\end{align}
where $\implies$ denotes weak convergence.
\end{thm}

\subsubsection*{The random walk}
The random walk $(X_t)$ is a process on $\integers^d$ whose transition probabilities  depend on the state of the contact process in a neighbourhood around the random walk. 
More precisely, we assume (w.l.o.g.) that $X_0=o$, where $o\in \Latd$ denotes the origin. Further, at any time $t>0$, the rate to jump from site $x$ to site $x + z$, given that the contact process is in state $\eta$ at time $t$, is given by $\gamma\alpha(\theta_x\eta, z) \in [0,\infty)$. Here, $\gamma\in [0,\infty)$ is a parameter of the model.

In order for the above process to be well defined, we need to pose some regularity assumptions. For this purpose, we assume throughout this paper that
\begin{align}\label{assump 1}
\norm{\alpha}_1 := \sum_{z \in \Latd} \norm{z} \sup_{\eta \in \Omega} |\alpha(\eta,z)| < \infty,
\end{align}
and that for some $R\in \nat$ and every $z\in \Latd$;
\begin{align}\label{assump 2}
\alpha(\eta,z)-\alpha(\omega,z) = 0 \text{ whenever } \eta \equiv \omega \text{ on } [-R,R]^d.
\end{align}
Assumption \eqref{assump 1} assures that the position of $(X_t)$ has a first moment, whereas Assumption \eqref{assump 2} says that  the random walk only depends on the contact process within a finite region around its location. Note that, by \eqref{assump 1},  the jump rate of $(X_t)$ is bounded by $\gamma \norm{\alpha}_1$.

Further, we say that the random walk is  \emph{elliptic} if there is a finite subset $E=\{e_1,\dots,e_n\}$ of $\Latd$ such that 
\begin{align}\label{eq weakly elliptic}
\alpha(\eta,e_i) >0 \quad \forall \: \eta \in \Omega \text{ and } i \in \{1,\dots,n\},
\end{align}
and such that $\alpha(\eta,y)>0$ for some $\eta \in \Omega$ and $y \in \Latd$ if and only if $ y=\sum_{i=1}^d a_i e_i$ with $a_i \in \{0,1,2,\dots\}$, $i=1,\dots,n$.

Lastly, for $\eta=(\eta_t) \in D_{\Omega}[0,\infty)$, let $P^{\eta}$ denote the \emph{quenched} law of $(X_t)$ in environment $\eta$. For $\mu \in \mathcal{M}_1(\Omega)$, the \emph{annealed}  law of $(X_t)$ is given by
\begin{align}
P^{\mu}(\cdot) := \int_{D_{\Omega}[0,\infty)} P^{\eta}(\cdot) \bP_{\mu,\lambda}(d\eta).
\end{align}

\subsubsection*{The contact process as seen from a random walk}

``The contact process  seen from the random walk'' (that is, the CPSRW process) is the key object of this paper. This process, which is also useful for understanding the asymptotic behaviour of the random walk itself, is the Markov process on $\Omega$ with generator
\begin{align}\label{eq generator CP}
L^{EP} f(\eta) := Lf(\eta) + \sum_{z\in \Latd} \alpha(\eta,z) \left[f(\theta_{-z}\eta)-f(\eta)\right],
\end{align}
corresponding semigroup $(S_t^{EP})$, both acting on $\mathcal{C}(\Omega;\reals)$, and with path-space measure denoted by $\bP_{\eta,\lambda}^{EP}$. Here, the superscript EP is an abbreviation for \emph{environment process} and is used to distinguish it from $\bP_{\eta,\lambda}$, the path-space measure of the contact process. 

\subsection{Main theorems}\label{sec main}

As for the ordinary contact process, it is clear that $\bar{0}$ is an absorbing state for the CPSRW process as well. If $\lambda < \lambda_c$ it is not difficult to show that $\delta_{\bar{0}}$ is the only stationary distribution for $(\eta_t^{EP})$. This follows for instance from the methods developed in \citet{RedigVolleringRWDRE2013} together with well known convergence estimates towards $\bar{0}$ for the subcritical contact process, see Theorem 1.2.48 in \cite{LiggettSIS1999}.

On the other hand, when $\lambda>\lambda_c$, one can often show that there exists more than one stationary distribution for the CPSRW process. For this, it is sufficient to show that there is a site $x \in \Latd$ such that 
\begin{align}\label{eq positive density} \limsup_{t \rightarrow \infty} \frac{1}{t} \int_0^t \eta_t^{EP}(x) dt >0.
\end{align}
That \eqref{eq positive density} holds  when $\lambda>\lambda_c$ can been shown by several methods. For instance, \cite[Theorem 1.4]{BethuelsenHeydenreichRWAD2015} and  \cite[Theorem 1]{HollanderSantosRWCP2013}, both proven via monotonicity arguments and particular properties of the contact process, imply that \eqref{eq positive density} holds for the class of models studied in these papers.  In \citet{SantosRWSEP2013} another method is put forward, by use of  multiscale analysis, and applied to a random walk on the exclusion process. This method can presumably be applied to random walks on the contact process as well.

Ideally we would like to describe the entire class of stationary distributions corresponding to $(\eta_t^{EP})$, given the transition kernel of $(X_t)$ and the infection parameter $\lambda$. As we saw in Theorem \ref{thm complete convergence}, a complete description is at hand for the ordinary contact process, i.e.\ when not perturbed by the random walk. Our main theorem shows that a similar statement holds for $(\eta_t^{EP}$)  when either $\lambda$ is sufficiently large or $\gamma$ is sufficiently small.

\begin{thm}[Complete convergence for $(\eta_t^{EP})$]\label{thm CE}
Assume that $(X_t)$ satisfies Assumptions \eqref{assump 1} and \eqref{assump 2} and that it is elliptic. 
\begin{description}
\item[a)]   Let $\lambda \in (\lambda_c,\infty)$. Then there is a $\gamma_0 \in (0,\infty)$ such that for all $\gamma <\gamma_0$ 
 there exists $\bar{\nu}_{\lambda}^{EP} \in \mathcal{M}_1(\Omega)$ making $\bP_{\bar{\nu}_{\lambda}^{EP},\lambda}^{EP}$ stationary and ergodic with respect to time-shifts. Furthermore, for any $\eta \in \Omega$,
\begin{align}
t^{-1} \int_0^t \delta_{\eta}S_s^{EP} ds \implies  \bP_{\eta,\lambda}(\tau =\infty)  \bar{\nu}_{\lambda}^{EP}
+ \bP_{\eta,\lambda}(\tau < \infty)\delta_{\bar{0}}.
\end{align}
\item[b)] Let $\gamma \in (0,\infty)$. Then there is a $\lambda_0 \in (0,\infty)$ such that for all $\lambda>\lambda_0$
there exists $\bar{\nu}_{\lambda}^{EP} \in \mathcal{M}_1(\Omega)$ making $\bP_{\bar{\nu}_{\lambda}^{EP},\lambda}^{EP}$ stationary and ergodic with respect to time-shifts. Furthermore, for any $\eta \in \Omega$,
\begin{align}
t^{-1} \int_0^t \delta_{\eta}S_s^{EP} \implies  \bP_{\eta,\lambda}(\tau =\infty)  \bar{\nu}_{\lambda}^{EP}
+ \bP_{\eta,\lambda}(\tau < \infty)\delta_{\bar{0}}.
\end{align}
\end{description}
\end{thm}

The choice of $\gamma_0$ and $\lambda_0$  in Theorem \ref{thm CE} is related to the asymptotic speed at which an infection spreads. That is,  we require the random walk trajectory to eventually be contained inside a forward space-time cone in which, for any starting configuration $\eta \in \Omega \setminus \{\bar{0}\}$, the contact process conditioned on survival is approximately in equilibrium. This is similar in spirit to the assumption on $\lambda$ in \cite[Theorem 2]{HollanderSantosRWCP2013}.

Similar perturbative regimes have recently been studied for several other RWDRE models with non-uniform dependence on the initial configuration, in particular by \citet*{AvenaBlondelFaggionatoRWDRE2016},
\citet*{HilarioHollanderSidoraviciusSantosTeixeiraRWRW2014} and \citet*{HuveneersSimenhausRWSEP2014}. These very interesting works do not overlap with that of this  paper and are furthermore based on very different methods. 

The proof of Theorem \ref{thm CE} follows by a coupling argument and uses known mixing properties of the supercritical contact process together with basic ergodic theory, and does not (directly) rely on the monotonicity properties of the contact process. Further, for what appears to be only due to technical matters, we restrict to convergence in the sense of Ces\`aro.

We further note that the strategy of the proof can be applied to other models with similar mixing properties as the contact process. For instance, Theorem \ref{thm CE} can be shown to hold for certain extensions of our model where the random walk is allowed to interact with the medium, i.e.\ the contact process, by locally adding/removing infections. Such extensions may be natural from an application point of view.

The ellipticity assumption in Theorem \ref{thm CE} seems necessary for the theorem to hold in general. Indeed, an example of an non-elliptic random walk for which there exists three extremal invariant measures for $(\eta_t^{EP})$ can be constructed by making the random walk resemble the behaviour of the rightmost particle process of the contact process on $\integers$.
One way to achieve this is by considering a random walk that jumps deterministically to the right at a rate $\gamma \geq \lambda$ when on an infected site and otherwise as a simple random walk with jump rate $1$. Considering the corresponding CPSRW process started from $\bar{0}$, $\bar{1}$ or the configuration where all sites on the negative integers are infected and the remaining sites are healthy, it is not difficult to show (using results about the distribution of the contact process seen from the rightmost particle, e.g.\  \citet{GalvesPresuttiEdgeCP1987}) that this process has  three invariant measures, all singular with respect to the other two.

Presumably, a similar reasoning can be made rigorous when $\lambda$ is close to the critical value or $\gamma$ is close to $0$, even in cases where $(X_t)$ is elliptic.
On the other hand, for the case considered in Theorem \ref{thm CE}, we do not think the ellipticity assumption is really necessary. We prove this rigorously in the case the contact process is started from $\bar{\nu}_{\lambda}$, as stated next. 

\begin{thm}[Convergence of the upper invariant measure]\label{thm CE3}
Assume that $(X_t)$ satisfies Assumptions \eqref{assump 1} and \eqref{assump 2}. Furthermore, let $\lambda$ and $\gamma$ be as in Theorem \ref{thm CE}.
  Then there exists $\bar{\nu}_{\lambda}^{EP} \in \mathcal{M}_1(\Omega)$ making $\bP_{\bar{\nu}_{\lambda}^{EP}}^{EP}$ stationary and ergodic with respect to time-shifts and such that
\begin{align} 
t^{-1} \int_0^t \bar{\nu}_{\lambda}S_s^{EP} ds \implies  \bar{\nu}_{\lambda}^{EP} \in \mathcal{M}_1(\Omega).
\end{align}
\end{thm}

As mentioned above, under fairly general assumptions on the random walk and assuming that either $\lambda$ is large or $\gamma$ is small, we believe that the CPSRW-process has exactly two extremal measure. That is,  in Theorems \ref{thm CE} and \ref{thm CE3}, we have $\bar{\nu}^{EP}_{\lambda} \neq \delta_{\bar{0}}$. Although we do not provide a proof of this, nevertheless, from the ergodic property of $\bar{\nu}_{\lambda}^{EP}$ only, we infer information about the random walk. 

\begin{thm}[Law of large numbers]\label{cor 2}
Under the  assumptions of Theorem \ref{thm CE} there exists $v_0,v_1 \in \reals^d$ such that for all $\eta \in \Omega$,
 \begin{align}\label{eq LLN} \lim_{t \rightarrow \infty} t^{-1} X_t = \bP_{\eta,\lambda}(\tau <\infty)v_0 + \bP_{\eta,\lambda}(\tau =\infty)v_1, \quad P^{\eta}-a.s.\end{align}
Relaxing the ellipticity assumption on $(X_t)$, \eqref{eq LLN} holds $P^{\bar{\nu}_{\lambda}}$-a.s.
\end{thm}

\begin{rem}\label{rem sunday}
Note that $\bP_{\eta}(\tau =\infty) =1$ if and only if $\eta$ has infinitely many infections, as follows by  \cite[Theorem 2.30]{LiggettSIS1999}. In particular, the limit in \eqref{eq LLN} equals $v_1 \in \reals^d$ when the contact process is started from the upper invariant measure $\bar{\nu}_{\lambda}$.
\end{rem}

Presumably, the law of large numbers in Theorem \ref{cor 2} can be extended to a functional central limit theorem under the annealed law. For this, from the existence of $\bar{\nu}_{\lambda}^{EP}$ that is ergodic under $(\eta_t^{EP})$, martingale methods (as used e.g.\ in \cite{RedigVolleringRWDRE2013}) seem useful.

To this end, a remark about the critical case (i.e., when $\lambda=\lambda_c$) is in place.
In this case, $\bar{\nu}_{\lambda_c} = \delta_{\bar{0}}$, as was proven by \citet{BezuidenhoutGrimmettCP1990}. This result has since been extended to several models.
On the other hand, still for $\lambda=\lambda_c$, 
the contact process process as seen from the rightmost particle is known to have a non-trivial invariant measure, as shown in \citet*{CoxDurrettSchinazi1991}.

It is not difficult to show that, by using Theorem 4.5 in \cite{BethuelsenVolleringRWDRE2016} and monotonicity of the contact process, for $\lambda=\lambda_c$, any invariant measure for the CPSRW process concentrates on configurations having $0$ asymptotic density. We believe that, under reasonable (ellipticity) assumptions, the CPSRW process with  $\lambda=\lambda_c$ has no non-trivial invariant measure. However, to show this rigorously seems challenging since the critical contact process has slowly decaying space-time correlation structure.

\section{Preliminaries about the contact process}\label{sec contact}

Important to our approach is the existence of a coupling $\hP_{\eta,\omega}^{\lambda}$ of the contact process started from any two  $\eta, \omega \in \Omega$. The canonical choice is the graphical construction coupling, see p.\ 32-34  in \cite{LiggettSIS1999}, however, any other coupling satisfying \eqref{eq coupling condition} and \eqref{eq coupling condition2} below will do just as fine. 

For $\eta,\omega \in \Omega$, the coupled pair $(\eta_t^1,\eta_t^2)_{t\geq 0}$ denotes two copies of the contact process, started from $\eta_0^1=\eta$ and $\eta_0^2=\omega$ respectively. Recall that, by definition, a coupling has the marginals
\begin{align}
\widehat{\bP}^{\lambda}_{\eta,\omega}(\eta_t^1\in\cdot)=\bP_{\eta,\lambda}(\eta_t\in\cdot) \text{ and }\widehat{\bP}_{\eta,\omega
}^{\lambda}(\eta_t^2\in\cdot)=\bP_{\omega,\lambda}(\eta_t\in\cdot).\end{align}
We are in this paper mainly interested in the contact process with $\lambda>\lambda_c$ for which $\bar{\nu}_{\lambda}$ is non-trivial. In this regime a more global description of the contact process is at hand and known as the shape theorem. For this, denote by $(\eta_t^o)$ and $(\eta_t^{\bar{1}})$ the contact process started from only the origin initially infected and the entire lattice initially infected respectively and define for $t\geq0$,
\begin{align}
&H_t := \{ x\in \Latd \colon \eta_s^o(x) = 1 \text{ for some } 0\leq s \leq t \}; 
 \\&K_t := \{ x\in \Latd \colon \eta_s^o(x) = \eta_s^{\bar{1}}(x) \: \forall \: s \geq t \}.
\end{align} 
$H_t$ is the set of sites which have been visited by an infection by time $t$ when the contact process is  started with only the origin infected at time $0$. $K_t$ is the subset of $\Latd$ where $(\eta_t^o)$ and $(\eta_t^{\Latd})$ remain coupled for all time after time $t$. 
The next result shows that when $\lambda>\lambda_c$, then $t^{-1} (H_t \cap K_t)$ has an asymptotic shape. To state the result it is convenient to consider
\begin{align}
\bar{H}_t := \bigcup_{x \in H_t} (x+Q) \text{ and } \bar{K}_t := \bigcup_{x \in K_t} (x+Q), \quad  Q= [-\frac{1}{2},\frac{1}{2}]^d.
\end{align}
Lastly, for $\omega \in \Omega$, denote by $\tau^{\omega} :=  \inf \{ t \geq 0 \colon \eta_t^{\omega} = 1_{\bar{0}} \}$ the time until the contact process started from $\omega$ ``dies out''.  We write $\tau^o :=  \inf \{ t \geq 0 \colon \eta_t^o = 1_{\bar{0}} \}$ for the case when $\omega(x)=1$ for $x=o$ only.

We are now prepared to state a version of the shape theorem, which in this generality is due to \citet{GaretMarchandShapeCPRE2012}; see Theorem 3 therein. 

\begin{thm}[The shape theorem]\label{thm shape}
Suppose $\lambda>\lambda_c$. There exists a convex set $D=D(\lambda) \subset \reals^d$ such that, for any $\epsilon>0$, on the event $\{\tau^o = \infty \}$,
\begin{align}\label{eq shape theorem}
\lim_{T \rightarrow \infty} \hP_{o,\bar{1}}^{\lambda} \left( (1-\epsilon)  D \subset \frac{1}{t}(\bar{H}_t \cap \bar{K}_t) \subset 
 \frac{1}{t} \bar{H}_t \subset (1+\epsilon)  D \: \forall t \geq T \right) =1.
\end{align}
Moreover, there is a function $f\colon (\lambda_c,\infty) \rightarrow (0,\infty)$, non-decreasing, and such that 
$\{x \in \reals^d \colon \norm{x}_1 \leq f(\lambda)\} \subset D(\lambda) \text{ and }\lim_{\lambda \rightarrow \infty}f(\lambda)=\infty.
$ 
\end{thm}

Theorem \ref{thm shape} implies mixing properties for the contact process when started from other configurations than only the origin initially infected, as we show next. 

\begin{lem}\label{lem CE}
Let $\lambda>\lambda_c$ and consider the contact processes $(\eta_t^{\eta})$ and $(\eta_t^{\bar{1}})$, initialised from $\eta$ and $\bar{1}$ respectively, where $\eta \in \Omega \setminus \{\bar{0}\}$. Then, for any $\epsilon>0$ and with $D$ as in Theorem \ref{thm shape},
\begin{align}\label{eq coupling condition}
 \lim_{T \rightarrow \infty} \hP_{\eta,\bar{1}}^{\lambda} \left(\eta_t^{\eta}(x) = \eta_t^{\bar{1}}(x)\:  \forall\: x \in t(1-\epsilon) D \: \forall \: t \geq T  \mid \tau^{\eta}=\infty \right) =1.
\end{align}
\end{lem}

\begin{proof}
Since the path measure of the contact process is translation invariant with respect to spatial shifts, \eqref{eq coupling condition} holds in the case when $\eta(x)=1$ for only one site $x \in \Latd$, as follows immediately from Theorem \ref{thm shape}. Indeed, for any $0< \epsilon_1< \epsilon_2$, it holds that $t(1-\epsilon_2)D \subset \theta_x t(1-\epsilon_1)D $ for all $t$ sufficiently large.

For $T \in (0,\infty)$, let
$A_T =\{\eta_t^{\eta}(x) \neq \eta_t^{\bar{1}}(x) \text{ for some } x \in t(1-\epsilon) D  \text{ and } t \geq T  \}$. Fix $x_1,\dots,x_n \in \Latd$ and assume that $\eta \in \Omega$ is such that $\eta(y)=1$ only when $y=x_1,\dots,x_n $. By using that the contact process is additive (in particular, that 
$\{ \tau^{\eta} <\infty\} =\cup_{i=1}^n  \{ \tau^{x_i}<\infty\} $), we have
\begin{align}
 \hP_{\eta,\bar{1}}^{\lambda} \left( A_T \mid \tau^{\eta}=\infty \right) 
 = & \hP_{\eta,\bar{1}}^{\lambda} \left( \tau^{\eta}=\infty \right)^{-1}  \hP_{\eta,\bar{1}}^{\lambda} \left(A_T, \tau^{\eta}=\infty \right)
 \\ = & \hP_{\eta,\bar{1}}^{\lambda} \left( \tau^{\eta}=\infty \right)^{-1}  \hP_{\eta,\bar{1}}^{\lambda} \left( A_T \cap \left(\cup_{i=1}^n  \{ \tau^{x_i}=\infty\} \right) \right)
 \\ \leq &  \hP_{\eta,\bar{1}}^{\lambda} \left( \tau^{\eta}=\infty \right)^{-1} \sum_{i=1}^n \hP_{\eta,\bar{1}}^{\lambda} \left( A_T \cap   \{ \tau^{x_i}=\infty\}  \right)
 \\  \leq & \sum_{i=1}^n \hP_{\eta,\bar{1}}^{\lambda} \left( A_T \mid   \{ \tau^{x_i}=\infty\}  \right).
 \end{align}
 Consequently,  \eqref{eq coupling condition} holds when $\eta$ has finitely many $1$'s, since  each term inside the sum in the last equation satisfies  \eqref{eq coupling condition}.

What remains to be shown is that \eqref{eq coupling condition}holds when the contact process is started from a configuration with infinitely many sites infected. In this case, $\tau^{\eta}=\infty$ a.s.\ (see Remark \ref{rem sunday}) and so 
\begin{align}
\hP_{\eta,\bar{1}}^{\lambda} \left( A_T \mid \tau^{\eta}=\infty \right)= \hP_{\eta,\bar{1}}^{\lambda} \left( A_T  \right).
\end{align}
Further, let $(N_n)_{n \in \nat}$ be a sequence such that $\sum_{x \in [-N_n,N_n]^d} \eta(x) \geq n$ and denote by $\sigma_n \in \Omega$ the configuration which equals $\eta$ on $[-N_n,N_n]^d$ and equals $0$ outside $[-N_n,N_n]^d$. Then 
\begin{align}
 \hP_{\eta,\bar{1}}^{\lambda} \left( A_T \right) &\geq \lim_{T \rightarrow \infty} \hP_{\eta,\bar{1}}^{\lambda} \left( A_T \cap \{ \tau^{\sigma_n}=\infty\}  \right) 
 \\ &= \hP_{\eta,\bar{1}}^{\lambda} \left( \tau^{\sigma_n} =\infty  \right)^{-1} \hP_{\eta,\bar{1}}^{\lambda} \left( A_T \mid \tau^{\sigma_n} =\infty \right).
\end{align}
Hence, taking $T \rightarrow \infty$ yields that $\lim_{T \rightarrow \infty} \hP_{\eta,\bar{1}}^{\lambda} \left( A_T \right) \geq \hP_{\eta,\bar{1}}^{\lambda} \left( \tau^{\sigma_n} =\infty  \right)^{-1}$.
Taking $n \rightarrow \infty$,  by  \cite[Theorem 1.2.30]{LiggettSIS1999}, we conclude the proof. 
\end{proof}

We also need to control the contact process started from a finite number of $1$'s and conditioned on dying out, for which we have the following lemma.

\begin{lem}\label{lem CE dies out}
Let $\lambda>\lambda_c$ and consider the contact processes $(\eta_t^{\eta})$ and $(\eta_t^{\bar{0}})$, initialised from $\eta$ and $\bar{1}$ respectively, where $\eta$ satisfies $\sum_{x\in \Latd} \eta(x)<\infty$. Then, for any $\epsilon>0$ and  with $D$ as in Theorem \ref{thm shape},
\begin{align}\label{eq coupling condition2}
 \lim_{T \rightarrow \infty} \hP_{\eta,\bar{0}}^{\lambda} \left(  \eta_t^{\eta}(x) = \eta_t^{\bar{0}}(x)  \: \forall\: x \in t(1-\epsilon) D \: \forall \: t \geq T \mid \tau^{\eta}<\infty \right) =1. 
\end{align}
\end{lem}

\begin{proof}
For $T \in (0,\infty)$, denote by
$A_T =\{ \exists x \in t(1-\epsilon) D \text{ with }t \geq T \colon \eta_t^{\eta}(x) =1  \}$. Fix $x_1,\dots,x_n \in \Latd$ and assume that $\eta \in \Omega$ is such that $\eta(y)=1$ only when $y=x_1,\dots,x_n $. By again using that the contact process is additive, we have that
\begin{align}
 \hP_{\eta,\bar{0}}^{\lambda} \left( A_T \mid \tau^{\eta}<\infty \right) 
 = & \hP_{\eta,\bar{0}}^{\lambda} \left( \tau^{\eta}<\infty \right)^{-1}  \hP_{\eta,\bar{0}}^{\lambda} \left(A_T, \tau^{\eta}<\infty \right)
 \\ = & \hP_{\eta,\bar{0}}^{\lambda} \left( \tau^{\eta}<\infty \right)^{-1}  \hP_{\eta,\bar{0}}^{\lambda} \left( A_T \cap \left(\cap_{i=1}^n  \{ \tau^{x_i}<\infty\} \right) \right).
 \end{align}
 By \cite[Theorem 1.2.30]{LiggettSIS1999}, the last equation decays exponentially in $T$ from which we conclude \eqref{eq coupling condition2}.
 \end{proof}

\section{Coupling construction}\label{sec coupling}

Given the coupling $\hP_{\eta,\omega}^{\lambda}$ of the contact process from the previous section, for each $T \in [0,\infty)$, we show in the following lemma how to extend it to a coupling $\hP_{\eta,\omega,T}^{\lambda}$ also containing the evolution of two random walks $(X_t^1,X_t^2)_{t\geq 0}$ on $(\eta_t^1,\eta_t^2)_{t\geq 0}$. The coupling construction is motivated by the coupling used in \cite{HollanderSantosRWCP2013}, Section 3, and can be seen as a generalisation of their approach to general dimensions and general transition kernels. 

Before stating the lemma we need to introduce some notation.
For $\gamma \in (0,\infty)$, denote by $\mathcal{R}(\gamma) \subset \reals^d$ the convex hull of the transition kernels of $(X_t)$,  that is,
\begin{align}\label{eq range}
\mathcal{R}(\gamma) := \gamma \cdot \text{conv}\left( \sum_{z \in \Latd} z \alpha(\eta,z), \eta \in \Omega \right).
\end{align}
Further, for $\eta,\omega \in \Omega$, $T,\epsilon>0$ and $n\in \nat$, let
\begin{align}\label{eq CTepsilon}
C_{T,n,\epsilon}(\eta,\omega) :=  \{ \eta_s^{\eta}(x) = \eta_s^{\omega}(x) \: \forall \: x \in s(1+\epsilon) \mathcal{R}(\gamma) +[-n,n]^d, s \in [T,\infty) \}
\end{align} denote the event that the contact processes started from $\eta$ and $\omega$ respectively are perfectly coupled inside $ s(1+\epsilon) \mathcal{R}(\gamma) +[-R,R]^d \subset \Latd$ for all $s \geq T$, and let
\begin{align}\label{eq DTepsilon}
D_{T,\epsilon} := \{  X_t^1, X_t^2 \in t(1+\epsilon) \mathcal{R}(\gamma) \: \forall \: t \geq T \}.
\end{align}

\begin{lem}\label{lem coupling}
Let $T \in [0,\infty)$ and let $\eta,\omega \in \Omega$.
There exists a coupling $\hP_{\eta,\omega,T}^{\lambda}$ with the following properties:

\begin{description}
\item[(a)] (Marginals) The coupling 
supports two contact processes and
corresponding random walks:
\begin{enumerate}
\item[1.] $\widehat{\bP}_{\eta,\omega,T}^{\lambda}((\eta^1_t,X_t^1)\in\cdot) = \tilde{\bP}
_{\eta,\lambda}((\eta_t,X_t) \in\cdot)$;
\item[2.] $\widehat{\bP}_{\eta,\omega,T}^{\lambda}((\eta^2_t,X_t^2)\in\cdot) = \tilde{\bP}
_{\eta,\lambda}((\eta_t,X_t) \in\cdot)$;
\end{enumerate}
where $\tilde{\bP}_{\eta,\lambda}$ is the path measure of the joint process $(\eta_t,X_t)$.
\item[(b)] (Extension of $\widehat{\bP}^{\lambda}_{\eta,\omega}$) The
contact processes behave
as under $\widehat{\bP}_{\eta,\omega}^{\lambda}$,
\[
\widehat{\bP}_{\eta,\omega,T}^{\lambda}\bigl(\bigl(\eta^1_t,
\eta_t^2\bigr)\in\cdot\bigr) = \widehat{\bP
}{}^{\lambda}_{\eta,\omega}\bigl(\bigl(\eta_t^1,
\eta_t^2\bigr) \in\cdot\bigr).
\]
\item[(c)] (Coupling of the walkers)
The jumping times of $X_t^1$ and $X_t^2$ are independent up to time $T$ and identical after time $T$.
Furthermore, for any $\epsilon>0$, \begin{align}
\lim_{T \rightarrow \infty} \hP_{\eta,\omega,T}^{\lambda} \left( X_t^1= X_t^2  \: \forall \: t \geq T \mid X_T^1=X_T^2, C_{T,R,\epsilon}(\eta,\omega), D_{T,\epsilon} \right) =1.
\end{align} 
\end{description}
\end{lem}

\begin{proof}

To obtain the properties listed above, we extend the original coupling $\widehat{\bP}{}_{\eta,\omega}^{\lambda}$
to contain three Poisson processes $N^{i}$ with $i \in \{1,2,3\}$,  all with rates $\lambda_i:= \gamma \norm{\alpha}_1$, as well as a
sufficient supply of independent uniform $[0,1]$ variables for each $i \in \{1,2,3\}$, denoted by $U^{i}$. The Poisson processes are chosen independent of $\widehat{\bP}{}_{\eta,\omega}^{\lambda}$ and thus property b) is immediate.  

To obtain the properties described in a) and c), we chose the Poisson processes $N^{1}$ and $N^{3}$ independent from each other, as well as the corresponding variables $U^{1}$ and $U^{3}$. Furthermore, the process $N^{2}$ is given by
\begin{align}
N_t^{2} := \left\{\begin{array}{cc} N_t^{3} & \text{if }t\leq T \\N_T^{3}+N_t^{1} - N_T^{1} & \text{if }t>T,\end{array}\right.
\end{align}
and the variables $U^{2}$ are given by
\begin{align}
U_n^{2} := \left\{\begin{array}{cc} U_n^{3}& \text{ if } n\leq N_T^{3}  \\ U_{n+N_T^1-N_T^3}^{1} & \text{ otherwise.}\end{array}\right.
\end{align}
Now, for $j\in \{1,2\}$, the random walk
$X^j$ starts from $o$ and exclusively (but not
necessarily) jump when the Poisson clocks $N^{j}$ rings. 
To make this precise, enumerate $\Latd = \{z_1,z_2,z_3,\dots\}$ and let for each $\eta \in \Omega$ and $m \in \nat$, $p(\eta,m ) := \sum_{i=1}^m \alpha(\eta,z_i)$.
When the
clock $N^{j}$ rings for the $k$'th time, the random walk jumps from $X^j_t$ to $X^j_t+z_i$ only if the
uniform $[0,1]$ variable $U^j_k$ satisfies 
\begin{align}
\norm{\alpha}_1^{-1} p(\theta_{X_t^j}\eta_t^j,z_{m-1}) \leq U_k^j< \norm{\alpha}_1^{-1} p(\theta_{X_t^j}\eta_t^j,z_{m}).
\end{align}
Clearly this yields property a). Furthermore, note that both the random walks use independent Poisson clocks and $U$'s up to time $T$, and share the same Poisson clocks and 
$U$'s after time $T$. Property c) follows as a consequence of this and since $(X_t)$ satisfies \eqref{assump 2}.
\end{proof}

\section{Proofs}\label{sec results Stein}

\subsection{Coupling argument}
In this subsection we present the coupling argument essential for the proofs of Theorems \ref{thm CE} and \ref{thm CE3}. For this, we first note that, as a simple consequence of our assumptions on the transition kernels of the random walk (recall Assumptions \eqref{assump 1}, \eqref{assump 2} and Definition \eqref{eq range}), the following lemma holds.

\begin{lem}\label{lem coupling help}
For any $\epsilon>0$ and any  $\eta,\omega \in \Omega$, it holds that
\begin{align}
\lim_{t \rightarrow \infty} \hP_{\eta,\omega,T} \left( X_s^1,X_s^2 \in s(1+\epsilon) \mathcal{R}(\gamma) \: \forall \: s \geq t \right) =1, \quad \text{ for any } T>0.
\end{align}
\end{lem}

With the help of  the coupling construction in the previous section, together with Lemma \ref{lem CE} and Lemma \ref{lem coupling help}, 
we next present a generalisation of  Proposition 3.3 in \cite{HollanderSantosRWCP2013} which allows us to compare possible limiting measures of $(\eta_t^{EP})$.

\begin{prop}\label{prop coupling}
Assume that $(X_t)$ satisfies Assumptions \eqref{assump 1} and \eqref{assump 2} and is elliptic. 
Furthermore, assume there exists $\mu^{EP} \in \mathcal{M}_1(\Omega)$ making $\bP_{\mu^{EP},\lambda}^{EP}$ stationary and ergodic with respect to time-shifts and such that $\mu^{EP}\neq \delta_{\bar{0}}$. 
If, for some $\epsilon>0$, $\mathcal{R}(\gamma)(1+\epsilon) \subset D$, then for every $f \in \cC(\Omega;\reals)$;
\begin{align}
\bP_{\omega,\lambda}^{EP} \left(\lim_{t\rightarrow \infty} t^{-1} \int_0^{t} f(\eta_s^{EP}) ds = \mu^{EP}(f) \mid \tau^{\omega}=\infty \right)=1.
\end{align}
\end{prop}

\begin{proof}
Let $\mu^{EP} \in \mathcal{M}_1(\Omega)$ be such that $\bP_{\mu^{EP},\lambda}^{EP}$ is ergodic with respect to time-shifts and $\mu^{EP}\neq \delta_{\bar{0}}$. By ergodicity, we know that there exist a set $B \in \mathcal{F}$ of full $\mu^{EP}$-measure such that for any $f \in \cC(\Omega;\reals)$ and $\eta \in B$; 
\begin{align}\label{eq ergodic}
\bP_{\eta,\lambda}^{EP} \left(\lim_{t\rightarrow \infty} t^{-1} \int_0^{t} f(\eta_s^{EP}) ds = \mu^{EP}(f) \right)=1.
\end{align}
Furthermore, since $\mu^{EP} \neq \delta_{\bar{0}}$ both are ergodic they are necessarily singular. 
Thus, $\mu^{EP}$ assigns $0$ probability to the class of configurations with only finitely many $1$'s. To see this, note that a configuration of finitely many $1$'s may die out in finite time with positive probability and,  since $\mu^{EP}$ and $\delta_{\bar{0}}$ are singular, this would lead to a contradiction. 
Consequently, by Remark \ref{rem sunday}, the ordinary contact process started from $\mu^{EP}$ survives almost surely.

Fix $\eta \in \Omega$ such that \eqref{eq ergodic} is satisfied. Let $\Lambda \subset [-n,n]^d$ for some $n \in \nat$ and consider a function $f \in \cC(\Omega;\reals)$ only depending on the configuration inside $\Lambda$. 
In order to prove Proposition \ref{prop coupling} we will show that for any $\omega \neq \bar{0}$;
\begin{align}\label{eq prop claim}
\bP_{\omega,\lambda}^{EP} \left(\lim_{t\rightarrow \infty} t^{-1} \int_0^{t} f(\eta_s^{EP}) ds = \mu^{EP}(f) \mid \tau=\infty \right)=1 \end{align} irrespectively of the choice of $\Lambda$ and $f$. This readily implies the statement of Proposition \ref{prop coupling} by standard arguments since $\cC(\Omega;\reals)$ is generated by the set of all local functions.

Now, to prove \eqref{eq prop claim} and complete the proof, consider the coupling $\hP_{\eta,\omega,T}^{\lambda}$ as constructed in Lemma \ref{lem coupling}. Let $M= \max(n,R)$, recall the definitions  \eqref{eq CTepsilon} and \eqref{eq DTepsilon},  and let  
\begin{align}
\Gamma_T := \{D_{T,\epsilon} \cap C_{T,M,\epsilon}(\eta,\omega) \text{ for some } \epsilon >0 \}.
\end{align}
Note that, by Lemma \ref{lem CE} and Lemma \ref{lem coupling help}, it holds that, for all $T>0$,
\begin{align}
\lim_{t \rightarrow \infty} \hP_{\eta,\omega,T}^{\lambda}(\Gamma_t\mid \tau^{\omega}=\infty)=1,\end{align} since by assumption $\mathcal{R}(\gamma)(1+\epsilon) \subset D$ for some $\epsilon>0$.
By the law of total expectation, by taking conditional expectation, \eqref{eq prop claim} thus follows if we can show that 
\begin{align}
\hP_{\eta,\omega,T}^{\lambda} \left(\lim_{t\rightarrow \infty} t^{-1} \int_0^{t} f(\theta_{X_s}\eta_s^2) ds = \mu^{EP}(f) \mid \Gamma_T,  \tau^{\omega}=\infty \right)=1, \quad \forall \: T>0.
\end{align}
Moreover, let $\epsilon>0$ be such that $\Gamma_T$ holds. Consequently, on $\Gamma_T$, $X_T^2=x$ for some $x \in \mathcal{R}(\gamma)(1+\epsilon)T$. 
Thus, by property c) of the coupling construction in Lemma \ref{lem coupling}, it suffices to show that
\begin{align}\label{eq prop conditioning}
\hP_{\eta,\omega,T}^{\lambda} \left(\lim_{t\rightarrow \infty} t^{-1} \int_0^{t} f(\theta_{X_s}\eta_s^2) ds = \mu^{EP}(f) \mid \Gamma_T,  \tau^{\omega}=\infty, X_T^2=x \right)=1, \quad \quad
\end{align} for all $T>0$ and $x \in \mathcal{R}(\gamma)(1+\epsilon)T$.

To this end, we employ the ellipticity assumption. For each $x\in \Latd$ fixed, there exists an event $B_x$ generated by $(N_{[0,T]}^1,U_{[1,N_{[0,T]}]})$ which has positive probability and such that $X_T^1 = x $ on $B_x$. By property c) of the coupling construction and due to the ellipticity assumption, $B_x$ can be chosen independent of the evolution of $(\eta_t^1,\eta_t^2)$ and $(X_t^2)$. 
Using this property, we thus have that
\begin{align}
&\hP_{\eta,\omega,T}^{\lambda} \left(\lim_{t\rightarrow \infty} t^{-1} \int_0^{t} f(\theta_{X_s^2}\eta_s^2) ds = \mu^{EP}(f) \mid \Gamma_T,  \tau^{\omega}=\infty, X_T^2=x \right) 
\\ =&\hP_{\eta,\omega,T}^{\lambda} \left(\lim_{t\rightarrow \infty} t^{-1} \int_0^{t} f(\theta_{X_s^2}\eta_s^2) ds = \mu^{EP}(f) \mid \Gamma_T,  \tau^{\omega}=\infty, X_T^2=x, B_x \right) 
\\ =&\label{eq prop conclude}\hP_{\eta,\omega,T}^{\lambda} \left(\lim_{t\rightarrow \infty} t^{-1} \int_0^{t} f(\theta_{X_s^1}\eta_s^1) ds = \mu^{EP}(f) \mid \Gamma_T,  \tau^{\omega}=\infty, X_T^2=x \right) \quad \quad \quad
\end{align} 
Here, the first equality holds since $B_{x}$ is independent of all the other variables. To see that the second equality holds, note that $\Gamma_T$ ensures that the contact processes are perfectly coupled inside the space-time region defined by $C_{T,M,\epsilon}(\eta,\omega)$. Furthermore, since $x\in \mathcal{R}(\gamma)(1+\epsilon)T$  and $X_T^1=X_T^2$ property c) of  the coupling construction apply.
To conclude \eqref{eq prop claim} and hence the proof of Proposition \ref{prop coupling}, we note that \eqref{eq prop conclude} equals $1$ as a consequence of \eqref{eq ergodic}. \end{proof}

By a straightforward adaptation of the proof of Proposition \ref{prop coupling}, replacing $\mu^{EP}$ by $\delta_{\bar{0}}$ and using Lemma \ref{lem CE dies out} instead of Lemma \ref{lem CE}, we have the following statement.

\begin{prop}\label{prop coupling3}
Assume that $(X_t)$ satisfies Assumptions \eqref{assump 1} and \eqref{assump 2} and is elliptic. 
Let $\omega \in \Omega$ be such that $\sum_{x\in \Latd}\omega(x)<\infty$. Then, for every $f \in \cC(\Omega;\reals)$;
\begin{align}
\bP_{\omega,\lambda}^{EP} \left(\lim_{t\rightarrow \infty} t^{-1} \int_0^{t} f(\eta_s^{EP}) ds = f(\bar{0}) \mid \tau^{\omega}<\infty \right)=1.
\end{align}
\end{prop}

Following \cite[Remark 3.4 ]{HollanderSantosRWCP2013}, the ellipticity assumption in the above argument is not necessary in the case  when the contact process is started from the upper invariant measure. The following proposition is essential for the proof of Theorem \ref{thm CE3}.

\begin{prop}\label{prop coupling2}
Assume that $(X_t)$ satisfies Assumptions \eqref{assump 1} and \eqref{assump 2}. 
Furthermore, assume there exists $\mu^{EP} \in \mathcal{M}_1(\Omega)$ making $\bP_{\mu^{EP},\lambda}^{EP}$ ergodic with respect to time-shifts and such that $\mu^{EP}\neq \delta_{\bar{0}}$. 
If, for some $\epsilon>0$, $\mathcal{R}(\gamma)(1+\epsilon) \subset D$, then for every $f \in \cC(\Omega;\reals)$;
\begin{align}
\bP_{\bar{\nu}_{\lambda},\lambda}^{EP} \left(\lim_{t\rightarrow \infty} t^{-1} \int_0^{t} f(\eta_s^{EP}) ds = \mu^{EP}(f) \mid \tau=\infty \right)=1.
\end{align}
\end{prop}

\begin{proof}
The statement follows as in the proof of Proposition \ref{prop coupling}, only with minor modifications which we highlight next. To adapt the proof, replace the conditioning on $X_T^2=x$ in \eqref{eq prop conditioning} (and the proceeding derivations) by the event $\{ N_T^2=0\}$, which implies $X_T^2=0$. Then, by stationarity of the contact process under $\bar{\nu}_{\lambda}$, $(X_{t+T}^2-X_T^2)_{t\geq0}$ under $\hP_{\eta,\bar{\nu}_{\lambda},T}^{\lambda}(\cdot \mid N_T^2=0)$ has the same distribution as $(X_t^2)_{t \geq 0}$ under $\hP_{\eta,\bar{\nu}_{\lambda},T}^{\lambda}(\cdot)$. Since $N_T^1=0$ implies $X_T^1=0$ and has positive probability, the claim follows as in the proof of Proposition \ref{prop coupling}  by replacing $B_x$ by $\{N_T^1=0\}$.
\end{proof}

\subsection{Proof of Theorems \ref{thm CE} and \ref{thm CE3} }

\begin{proof}[Proof of Theorem \ref{thm CE}]
As mentioned in the introduction, the measure $\delta_{\bar{0}}$ is trivially an invariant measure for the CPSRW process. Furthermore, it clearly makes $\bP_{\delta_{\bar{0}}}^{EP}$ ergodic with respect to time-shifts and is hence extremal.
Thus, in the (unlikely) scenario that $\delta_{\bar{0}}$ is the unique invariant measure for the CPSRW process, Theorem \ref{thm CE} follows by classical ergodic theory with $\bar{\nu}_{\lambda}^{EP}=\delta_{\bar{0}}$.

To complete the argument of Theorem \ref{thm CE}, we next consider the (more likely) scenario  that there exist a measure $\mu^{EP} \in \mathcal{M}_1(\Omega)$ invariant under $(\eta_t^{EP})$ and such that $\mu^{EP} \neq \delta_{\bar{0}}$. Without loss of generality, assume that $\mu^{EP}$ is extremal and hence singular with respect to $\delta_{\bar{0}}$. By Proposition \ref{prop coupling} together with Remark \ref{rem sunday}, the statement of Theorem \ref{thm CE} follows immediately when starting the CPSRW process with a configuration having infinitely many $1$'s in the case that $\mathcal{R}(\gamma)(1+\epsilon) \subset D$. In this case the CPSRW process  convergence towards $\bar{\nu}_{\lambda}^{EP}$.  For fixed $\lambda>\lambda_c$, the statement of Theorem \ref{thm CE}a) thus follows by taking $\gamma$ sufficiently small. Similarly, for fixed $\gamma\in (0,\infty)$, the statement of Theorem \ref{thm CE}b) follows by taking $\lambda$ sufficiently large, since $D=D(\lambda)$ is growing towards the whole lattice as $\lambda$ increases.

Similarly, if the starting configuration $\eta\neq \bar{0}$ has only finitely many $1$'s, the CPSRW process converges towards $\bar{\nu}_{\lambda}^{EP}$ on the event that $\{\tau=\infty\}$. This follows again  by applying Proposition \ref{prop coupling}. On the other hand, on the event $\{\tau<\infty\}$, by Proposition \ref{prop coupling3}, the CPSRW process converges towards $\delta_{\bar{0}}$. This concludes the proof.
\end{proof}

\begin{proof}[Proof of Theorem \ref{thm CE3}]
This follows analogous to the proof of Theorem \ref{thm CE} for the case when $\eta \in \Omega$ has infinitely many $1$'s, by applying Proposition \ref{prop coupling2} instead of Proposition \ref{prop coupling}.
\end{proof}

\subsection{Proof of Theorem \ref{cor 2}}

\begin{proof}[Proof of Theorem \ref{cor 2}]
For the first part of the proof, we follow the proof of \cite{RedigVolleringLTRWDRE2011}, Theorem 4.1. 
Consider $(\eta_t^{EP})$ started from $\bar{\nu}_{\lambda}^{EP}$. Let $F_z : D([0,1],\Omega)\mapsto \bR$, $z\in\bZ^d$ count the number of shifts of size $z$ a piece of trajectory performs in the interval $[0,1]$, i.e.
\begin{align} F_z(\eta^{EP}_{[t,t+1]}) = \suml_{s\in]0,1]} \ind_{\{ \theta_{z}\eta^{EP}_{t+s}=\eta^{EP}_{t+s-} \} }. \end{align}
With $F := \sum_{z\in\bZ^d}zF_z$, which is well-defined and in $L^1(\mu^{EP})$ because of the rate condition $\norm{\alpha}_1<\infty$, we then have, for any integer $T>0$,
\[ X_T - X_0 = \sum_{n=1}^T F(\eta^{EP}_{[n-1,n]}). \]
The ergodic theorem then implies
\begin{align}\label{eq LLN convergence}
 \lim_{T\to\infty} \frac{X_T - X_0}{T} = \lim_{T\to\infty}\frac{1}{T}\sum_{n=1}^T F(\eta^{EP}_{[n-1,n]}) = \bar{\nu}_{\lambda}^{EP}(F). \end{align}
The same is true for non-integer $T$, by using the fact that $X_T-X_{\lfloor T \rfloor}$ has bounded expectation.
Since 
\begin{align*}
 \bar{\nu}_{\lambda}^{EP}(F) &= \int \int_0^1 \bE_{\eta,\lambda} \suml_{z\in\bZ^d} z\alpha(\theta_{X_t}\eta_t,z)\,dt\,\mu^{EP}(d\eta)\\
 & = \int \suml_{z\in\bZ^d} z\alpha(\eta,z)\,\mu^{EP}(d\eta),
\end{align*}
the claim is proven for $\bar{\nu}_{\lambda}^{EP}$.

To extend the result to an arbitrary probability measure $\nu$ we use Theorem \ref{thm CE}. Firstly, if $\bar{\nu}_{\lambda}^{EP} = \delta_{\bar{0}}$, then it follows that $(\eta_t^{EP})$ converges towards $\delta_{\bar{0}}$ when started from any $\eta_0 \in \Omega$. Consequently, the left hand side of \eqref{eq LLN convergence} converges towards $F(\bar{0})$ irrespectively of $\eta_0 \in \Omega$. 

Secondly,  if $\bar{\nu}_{\lambda}^{EP} \neq \delta_{\bar{0}}$, we concluded in the proof of Theorem \ref{thm CE} that $\bar{\nu}_{\lambda}^{EP}$ concentrates on configurations which have infinitely many infections. For any such configuration we showed in Lemma \ref{lem CE} that \eqref{eq coupling condition} holds. Thus, for any $\eta \in \Omega$, the left hand side of \eqref{eq LLN convergence} converges towards $\bar{\nu}_{\lambda}^{EP}(F)$ when conditioned on $\tau^{\eta}=\infty$. Similarly, on $\tau^{\eta}<\infty$, \eqref{eq LLN convergence} converges towards $F(\bar{0})$, and this concludes the proof.
\end{proof}




\end{document}